\documentclass[11pt,reqno]{amsart}
\usepackage{amsaddr}

\usepackage{a4wide}
\usepackage{amsthm}
\usepackage{amsmath}
\usepackage{amssymb}
\usepackage[all,cmtip]{xy}
\usepackage{color}
\usepackage{enumerate}
\usepackage{verbatim}

\usepackage[utf8]{inputenc}

\theoremstyle{plain}
\newtheorem{thm}{Theorem}
\newtheorem{lem}[thm]{Lemma}
\newtheorem{prop}[thm]{Proposition}
\newtheorem{cor}[thm]{Corollary}
\theoremstyle{definition}
\newtheorem{defn}[thm]{Definition}
\newtheorem{exmp}[thm]{Example}

\newtheorem{rem}[thm]{Remark}

\DeclareMathOperator{\id}{id}

\newcommand{\N}{\mathbb{N}}
\newcommand{\Z}{\mathbb{Z}}

\newcommand{\R}{\mathbb{R}}

\begin{document}

\title[Non-unital Ore extensions]{Non-unital Ore extensions}


\author{Patrik Lundstr\"{o}m}
\address{Department of Engineering Science,
University West,
SE-46186 Trollhättan, Sweden}

\author{Johan \"{O}inert}
\address{Department of Mathematics and Natural Sciences,
Blekinge Institute of Technology,
SE-37179 Karlskrona, Sweden}

\author{Johan Richter}
\address{Department of Mathematics and Natural Sciences,
Blekinge Institute of Technology,
SE-37179 Karlskrona, Sweden}

\email{patrik.lundstrom@hv.se; johan.oinert@bth.se; johan.richter@bth.se}


\begin{abstract}
In this article, we study Ore extensions of non-unital associative rings.
We provide a characterization of simple non-unital differential polynomial rings $R[x;\delta]$,
under the hypothesis that $R$ is $s$-unital and $\ker(\delta)$ contains a nonzero idempotent.
This result generalizes a result by \"Oinert, Richter and Silvestrov from the unital setting.
We also present a family of examples of simple non-unital differential polynomial rings.
\end{abstract}

\maketitle

\pagestyle{headings}


\section{Introduction}
In this article all rings will be associative, but not necessarily unital.
Ore extensions are unital rings that were introduced by Ore \cite{ore1933} under the name of \emph{non-commutative polynomial rings}.
Our aim here is to study a non-unital generalization of Ore extensions.

Let $R$ be a unital ring,
let $\sigma : R \to R$ be a unital ring endomorphism (not necessarily injective), and let $\delta : R \to R$ be a $\sigma$-derivation, i.e.
$\delta$ is an additive map satisfying
\begin{displaymath}
	\delta(rs)=\sigma(r)\delta(s) + \delta(r)s
\end{displaymath}
for all $r,s \in R$.
The \emph{Ore extension} $R[x;\sigma,\delta]$
is defined as the ring generated by $R$ and an element $x\notin R$
such that $1,x,x^2,\ldots$ form a basis for $R[x;\sigma,\delta]$ as a left $R$-module and all $r\in R$ satisfy
\begin{equation}
xr = \sigma(r)x + \delta(r).
\label{eq:DefRelation}
\end{equation}
Such a ring always exists and is unique up to isomorphism (see \cite{GoodearlWarfield}).
Since $\sigma(1_R) = 1_R$ and $\delta(1_R) = \delta(1_R \cdot 1_R) = \sigma(1_R) \cdot \delta(1_R)+ \delta(1_R) \cdot 1_R$, one gets that $\delta(1_R) = 0$ and hence $1_R$ will
be a multiplicative identity element for $R[x; \sigma, \delta]$ as well.

The Ore extensions
play an important role when
investigating cyclic algebras, enveloping rings of solvable Lie algebras,
and various types of graded rings 
such as group rings and crossed products,
see e.g. \cite{cohn1977}, \cite{jacobson1999},
\cite{mcconell1988} and \cite{rowen1988}.
They are also a natural source of
examples and counter-examples in ring theory,
see e.g. \cite{bergman1964} and \cite{cohn1961}.

Ore extensions $R[x; \id_R, \delta]$ with the endomorphism being equal to the identity map are called \emph{differential polynomial rings} and are denoted by $R[x;\delta]$. In that case, $\delta$ is a derivation.

\"Oinert, Richter and Silvestrov \cite{ORS2013} investigated when Ore extensions are simple. In particular, they obtained necessary and sufficient conditions for differential polynomials rings to be simple. For the purposes of this article, the following is the most important of their results.

\begin{thm}[{\cite[Theorem 4.15]{ORS2013}}]\label{thm:unitalsimplicy}
If $R$ is a unital ring with a derivation $\delta$, then $S:=R[x;\delta]$ is simple if and only if $R$ is $\delta$-simple and $Z(S)$ is a field. 
\end{thm}

B\"ack, Richter and Silvestrov \cite{BRS2018} discussed Ore extensions of non-unital rings in the hom-associative context, including Ore extensions of non-associative rings. In \cite{BRS2018},  they proved well-definedness of the construction using direct computations, without passing to unitalizations as we do in the present article (cf. Section~\ref{sec:unitalization}). 

Nystedt, \"Oinert and Richter \cite{NOR2018,NORmonoid} generalized the construction of Ore extensions to obtain non-associative Ore extensions, and the even broader class of \emph{Ore monoid rings}, and generalized the conditions on simplicity from \cite{ORS2013}.

Non-unital rings arise naturally in e.g. functional analysis. One example coming from distribution theory, is the ring $C^{\infty}_c(\mathbb{R})$ of smooth functions with compact support. 
The rings $C_0(X)$ of continuous functions that vanish at infinity, where $X$ is a locally compact Hausdorff space, 
are further examples of great importance.
A well-known theorem in the theory of operator algebras asserts that every commutative $C^{*}$-algebra is isomorphic to some $C_0(X)$. Unless $X$ is compact, $C_0(X)$ will in fact be a non-unital ring.

Many important non-unital rings satisfy weaker analogues of unitality.
Indeed, one can show (see e.g. \cite{Ny2019}) that the following chain of inclusions hold for
different classes of rings:
\[
  \{ \mbox{unital rings} \} \subsetneq 
  \{ \mbox{rings with enough idempotents} \} \subsetneq
  \{ \mbox{locally unital rings} \} 
\]
\[
  \subsetneq
  \{ \mbox{$s$-unital rings} \} \subsetneq 
  \{ \mbox{idempotent rings} \}. 
\]
Recall that a ring $R$ is called {\it idem\-potent} if $R R = R$.
Following Fuller \cite{Fu}, we say that $R$ has {\it enough idempotents} if there exists a set 
$\{ e_i \}_{i \in I}$ of orthogonal idempotents in $R$ (called a complete set of idempotents for $R$) 
such that $R = \bigoplus_{i\in I} R e_i = \bigoplus_{i \in I} e_i R.$
Following \'{A}nh and M\'{a}rki \cite{anhmarki1987}, we say that $R$ is {\it locally unital} if for all $n \in \N$ and all
$r_1,\ldots , r_n \in R$ there is an idempotent $e \in R$ such that for all $i \in \{1, \ldots , n\}$
the equalities $e r_i = r_i e = r_i$ hold.
The ring $R$ is called {\it $s$-unital} if for all $r \in R$ the relation $r \in Rr \cap rR$ holds.

A concrete example of a non-unital ring with enough idempotents is $M_\infty(U)$, 
the ring of infinite matrices over a unital ring $U$, where each matrix only has 
a finite number of nonzero entries. This ring is clearly locally unital.
Locally unital rings appear frequently in mathematics. 
The class of locally unital rings include the von Neumann regular rings \cite{anhmarki1987}, and Leavitt path algebras \cite{AbramsAraMolineBook}.
Furthermore, rings of functions with compact support (cf.~ Example~\ref{ex:outercomm}) are locally unital.  
On the other hand, it is easy to see that if we consider rings of \emph{continuous} functions with compact support,
then such a ring is always $s$-unital but locally unital only when the underlying space is compact (in which case the ring
is in fact unital).

In this article, we are going to generalize the above construction of Ore extensions to the non-unital setting.
Indeed, we are going to allow $R\neq \{0\}$ to be an arbitrary (not necessarily unital) associative ring. 
To be precise, let $R[x;\sigma,\delta]$ denote the set $\{ r_0 + \sum_{i=1}^n r_i x^i \mid n\geq 1,\quad r_0,\ldots,r_n \in R\}$.
It is clear that $R[x;\sigma,\delta]$ has a natural left $R$-module structure. 
Note that symbols of the form $x^i$ are usually not elements of $R[x;\sigma,\delta]$. Instead, we will mainly view them as placeholders. To ease notation we will, however, sometimes write $r x^0$ instead of an actual element $r \in  R[x;\sigma,\delta]  \cap R$.

We will define a ring multiplication on $R[x;\sigma,\delta]$ in the following way.
For $n,m>0$ and all $a,b\in R$, we put
\begin{eqnarray}
    (a x^n) (b x^m) &:=& \sum_{i=0}^n a\pi_i^n(b) x^{i+m}\\
    a(bx^m) &:=& abx^m, \text{ and}\\
    (ax^n)b &:=& \sum_{i=0}^n a \pi_i^n(b) x^i. 
\end{eqnarray}
Here $\pi_i^n$ denotes the sum of all $\binom{n}{i}$ possible compositions of $i$ copies of $\sigma$ and $n-i$ copies of $\delta$ in arbitrary order. For instance, $\pi_1^2=\sigma\circ\delta+\delta\circ\sigma$, whereas $\pi_0^0=\mathrm{id}_R$. In the special case of differential polynomial rings, we get

\begin{eqnarray}
	(a x^n) (b x^m) &=& \sum_{i=0}^n \binom{n}{i} a \delta^{n-i}(b) x^{i+m}, \label{eqn:prod1}\\
	a (b x^m) &=& ab x^m, \text{ and} \label{eqn:prod2}\\
	(a x^n) b &=& \sum_{i=0}^n \binom{n}{i} a \delta^{n-i}(b) x^i. \label{eqn:prod3}
\end{eqnarray}

Similarly to the classical (unital) case,
if $\sigma = \id_R$, then we will refer to $R[x;\id_R,\delta]$ as a \emph{non-unital differential polynomial ring} and simply denote it by $R[x;\delta]$.
If $\delta \equiv 0$, then $R[x;\sigma,0]$ is called a \emph{non-unital skew polynomial ring}.

This article is organized as follows.
In Section~\ref{sec:unitalization}, we will show that if $R$ is a non-unital ring, $\sigma$ is an endomorphism on $R$, and $\delta$ is a $\sigma$-derivation, then $\sigma$ and $\delta$ can be extended to an endomorphism respectively a $\sigma$-derivation on $R'$, the unitalization of $R$ (see Theorem~\ref{thm:unitalization}). This result is used to establish that the non-unital Ore extension is a well-defined associative ring (see Corollary~\ref{cor:associative}).
In Section~\ref{sec:simplicity}, we study conditions under which a non-unital Ore extension is simple, especially in the case of differential polynomial rings. 
In particular, we generalize Theorem~\ref{thm:unitalsimplicy} to $s$-unital differential polynomial rings 
(see Theorem~\ref{thm:simplicitynonunitality}).
In Section~\ref{sec:ex}, several examples are presented.

\section{The Ore extension of a unitalized ring}\label{sec:unitalization}

In this section we will establish that non-unital Ore extensions are associative (see Corollary~\ref{cor:associative}).
Recall that the \emph{unitalization} $R'$ of a ring $R$ is defined as the set $R':=R \times \Z$.
The addition on $R'$ is defined component-wise and the multiplication on $R'$ is defined by the rule
$(r,n)(s,m)=(rs+ns+mr,nm)$, for $(r,n),(s,m) \in R'$.
Note that $(0,1_R)$ is the multiplicative identity element of $R'$.

\begin{thm}\label{thm:unitalization}
Let $R$ be a non-unital ring and let $R'$ be its unitalization. If $\sigma : R \to R$ is a ring endomorphism and $\delta : R \to R$ is a $\sigma$-derivation, then they both have extensions $\tilde{\sigma}$ and $\tilde{\delta}$ to $R'$, such that $\tilde{\sigma}$ is a ring endomorphism of $R'$, respecting the multiplicative identity element, and $\tilde{\delta}$ is a $\tilde{\sigma}$-derivation. 
\end{thm}

\begin{proof}
For any $(r,n) \in R'$, we define $\tilde{\sigma}(r,n) = (\sigma(r),n)$ and $\tilde{\delta}(r,n) =(\delta(r),0)$. We first show that $\tilde{\sigma}$ is a ring endomorphism of $R'$ which respects the
multiplicative identity element.

Clearly, $\tilde{\sigma}$ is additive.
For all $(r,n), (s,m) \in R'$, we get that 
\begin{align*}
\tilde{\sigma}((r,n)(s,m)) &= \tilde{\sigma}(rs+ns+mr,nm) =(\sigma(rs+ns+mr),nm) \\
&=(\sigma(rs)+n\sigma(s)+m\sigma(r),nm)=(\sigma(r)\sigma(s)+n\sigma(s)+m\sigma(r),nm) \\
&=(\sigma(r),n)(\sigma(s),m) = \tilde{\sigma}(r,n) \tilde{\sigma}(s,m).
\end{align*}

\noindent Thus, $\tilde{\sigma}$ is multiplicative. Moreover, $\tilde{\sigma}(0,1_R)=(\sigma(0),1_R)=(0,1_R)$.  
 
It remains to show that $\tilde{\delta}$ is a  $\tilde{\sigma}$-derivation. Clearly, $\tilde{\delta}$ is additive. For all $(r,n), (s,m) \in R'$, we get that
\begin{align*}
\tilde{\delta}((r,n)(s,m)) &=  \tilde{\delta}(rs+ns+mr,nm) = (\delta(rs+ns+mr),0) \\
& = (\delta(rs) +n\delta(s)+m\delta(r),0) = (\sigma(r)\delta(s)+\delta(r)s+n\delta(s)+m\delta(r),0) \\
& = (\sigma(r)\delta(s)+n\delta(s),0)+(\delta(r)s+m\delta(r),0) \\
& = (\sigma(r),n)(\delta(s),0)+(\delta(r),0)(s,m)
= \tilde{\sigma}(r,n)\tilde{\delta}(s,m)+\tilde{\delta}(r,n) (s,m). \qedhere
\end{align*} 
\end{proof}

\begin{cor}\label{cor:associative}
Every non-unital Ore extension $R[x,\sigma,\delta]$
is an associative ring.
\end{cor}

\begin{proof}
If $R$ is a non-unital ring, $\sigma$ is a ring endomorphism of $R$ and $\delta$ is a $\sigma$-derivation on $R$, then we can extend $\sigma$ and $\delta$ to the unitalization $R'$ and form the Ore extension $R'[x; \tilde{\sigma}, \tilde{\delta}]$. The non-unital Ore extension $R[x,\sigma,\delta]$ is naturally embedded as an ideal in the associative ring $R'[x; \tilde{\sigma}, \tilde{\delta}]$. For a proof of the associativity of $R'[x; \tilde{\sigma}, \tilde{\delta}]$, see e.g. \cite{nystedt2013}.
\end{proof}

\section{Simplicity of non-unital differential polynomial rings}\label{sec:simplicity}

In this section we will give a characterization of simple non-unital differential polynomial rings (see Theorem~\ref{thm:simplicitynonunitality}).

\begin{prop}\label{prop:idempotents}
Let $R$ be a ring and let $\delta : R \to R$ be a derivation.
If $e\in R$ is an idempotent, then the following five assertions hold:
\begin{enumerate}[{\rm (a)}]
	\item $e \delta(e) e = 0$;
	\item If $e\in Z(R)$, then $\delta(e)=0$;
	\item $I=ReR$ is a $\delta$-invariant ideal of $R$.
	\item If $\delta(e)=0$, then $\delta(er) = e\delta(r)$ and $\delta(re) =\delta(r)e$
	for every $r \in R$. In particular, $\delta(ere) = e\delta(r)e$ for every $r \in R$. 
	\item If $\delta(e)=0$, then the restriction of $\delta$ to $eRe$ is a derivation of $eRe$. If $\delta$ is an inner derivation, then so is its restriction to $eRe$. 	
\end{enumerate}
\end{prop}

\begin{proof}
(a) Using that $\delta$ is a derivation, we get that
$\delta(e)=\delta(ee)=e\delta(e)+\delta(e)e$.
Thus, $e\delta(e)e=e^2\delta(e)e + e\delta(e)e^2 = 2\cdot e\delta(e)e$.
This shows that $e\delta(e)e=0$.

(b) Suppose that $e\in Z(R)$. By (a) we get that
$\delta(e) e = e \delta(e) = e^2 \delta(e) = e\delta(e)e=0$.
Using this, we notice that
$\delta(e)=\delta(e^2)=e\delta(e) + \delta(e)e = 0 + 0 = 0$.

(c) It is clear that $I$ is an ideal of $R$.
For any elements $r,s\in R$ we get that 
\begin{displaymath}
	\delta(es) =\delta(ees) = e\delta(es)+\delta(e)es = ee\delta(es)+\delta(e)es \in I
\end{displaymath}
and hence
\begin{displaymath}
	\delta(res) = r\delta(es)+\delta(r)es \in I.
\end{displaymath}
This shows that $\delta(I) \subseteq I$.

(d) Suppose that $\delta(e)=0$.
Take $r\in R$.
Then
$\delta(er)= e\delta(r)+\delta(e)r = e\delta(r)$ and $\delta(re) = r\delta(e)+\delta(r)e =\delta(r)e$. (Note that $e$ need not be an idempotent for the proof to work.)

(e) Suppose that $\delta(e)=0$. By (d), $\delta$ maps $eRe$ into $eRe$. Thus, the restriction of $\delta$ to $eRe$ is clearly a derivation. If we suppose that $a\in R$ is such that $\delta(r) = ar-ra$ for all $r\in R$, then for $ere \in eRe$ we get
\begin{displaymath}
\delta(ere) = e\delta(ere)e = eaere-ereae =(eae)(ere)-(ere)(eae).
\end{displaymath}
Hence, the restriction of $\delta$ to $eRe$ is the inner derivation on $eRe$ induced by $eae$. 
\end{proof}

\begin{rem}
If an idempotent $e$ is not central, then $\delta(e)=0$ need not hold.
Indeed, consider the matrix ring $R:=M_2(\R)$ of 2-by-2 matrices over the real numbers.
The matrix $e = \left(\begin{smallmatrix}1&1\\0&0\end{smallmatrix}\right)$
is idempotent.
We define a derivation $\delta : R \to R$
by $\delta(M):=\left(\begin{smallmatrix}0&0\\1&0\end{smallmatrix}\right) M - M \left(\begin{smallmatrix}0&0\\1&0\end{smallmatrix}\right)$,
for $M \in R$.
An easy calculation shows that
$\delta(e)=\left(\begin{smallmatrix}-1&0\\1&1\end{smallmatrix}\right) \neq \left(\begin{smallmatrix}0&0\\0&0\end{smallmatrix}\right)$.
\end{rem}

\begin{defn}\label{def:invariant}
Let $R[x;\sigma,\delta]$ be a non-unital Ore extension.
An ideal $J$ of $R$ is said to be \emph{$\sigma$-$\delta$-invariant} if $\sigma(J)\subseteq J$ and
$\delta(J) \subseteq J$. If $\{0\}$ and $R$ are the only $\sigma$-$\delta$-invariant ideals of $R$, then $R$ is called \emph{$\sigma$-$\delta$-simple}.
\end{defn}

\begin{prop}\label{sigmadeltasimple}
If $R[x;\sigma,\delta]$ is a simple non-unital Ore extension,
then $R$ is $\sigma$-$\delta$-simple.
\end{prop}

\begin{proof}
Suppose that $R[x;\sigma,\delta]$ is  simple.
Take a nonzero $\sigma$-$\delta$-invariant ideal $J$ of $R$.
We wish to show that $J = R$.
Let $I$ be the additive subgroup of $R[x;\sigma,\delta]$ 
consisting of finite sums $j_0 + \sum_{k=1}^n j_k x^k$ where $j_0,j_1,\ldots,j_n \in J$.
Clearly, $I$ is nonzero.
By the product rule (see \eqref{eqn:prod1}--\eqref{eqn:prod3}) it is clear that $I$ is a right ideal of $R[x;\sigma,\delta]$.
From the $\sigma$-$\delta$-invariance of $J$, and the product rule, it follows
that $I$ is also a left ideal of $R[x;\sigma,\delta]$.
By the simplicity of $R[x;\sigma,\delta]$, we get that $I=R[x;\sigma,\delta]$.
Thus, $J = R$. This shows that $R$ is $\sigma$-$\delta$-simple.
\end{proof}

\begin{rem}\label{rem:localunits}
Let $R$ be an arbitrary ring.

(a) If $e$ is an idempotent of $R$, 
then the set $eRe$ is called a \emph{corner subring} of $R$ and $e$ is its multiplicative identity element.

(b) Recall that $R$ is said to be \emph{locally unital}, if for each finite set $F \subseteq R$ there is an idempotent $e\in R$ such that $F \subseteq eRe$.
In that case, $ex=xe=x$ for each $x\in F$, and $e$ is referred to as a \emph{local unit} for the set $F$.

(c) If $R$ is locally unital, then a set $E\subseteq R$ of idempotents which constitute local units of $R$, is said to be a \emph{set of local units} for $R$.
In that case, $E=E(R)$ (the set of all idempotents of $R$) is obviously a set of local units.
However, it may also be possible to exclude som idempotents and choose $E$ to be a smaller set (see e.g. Example~\ref{ex:nonunital}).
\end{rem}

\begin{prop}\label{prop:cornerIso}
Let $S:=R[x; \delta]$ be a non-unital differential polynomial ring. If $e\in R$ is a nonzero idempotent such that $\delta(e) = 0$, then the corner subring $eSe$ is isomorphic to the unital Ore extension $eRe[x;d]$, where $d$ is the restriction of $\delta$ to $eRe$. 
\end{prop}

\begin{proof}
By Proposition~\ref{prop:idempotents}, $d$ is a derivation on the unital ring $eRe$. The elements of $eSe$ are of the form $\sum_i ea_i x^i e$. Using the fact that $\delta(e)=0$ we can rewrite them to the form $\sum_i ea_i e x^i$. The conclusion follows immediately. 
\end{proof}

\begin{prop}\label{prop:simpleCorner}
If $S$ is a simple ring,
then for every nonzero idempotent $e\in S$
the corner subring $eSe$ is also simple. In particular, $Z(eSe)$ is a field. 
\end{prop}

\begin{proof}
Suppose that $S$ is a simple ring and let $e \in S$ be a nonzero idempotent. Notice that $eSe$ is nonzero since $e=eee\in eSe$.
Let $I$ be a nonzero ideal of $eSe$. We will show that $I=eSe$.
Take a nonzero $x \in I \subseteq eSe$.
Notice that $x=exe \in SxS$.
By simplicity of $S$, we get that $SxS+\Z x = SxS =S$.
In particular, we may write
$e = \sum_{i=1}^n y_i x z_i$ where $y_1,\ldots,y_n, z_1,\ldots,z_n$ all belong to $S$.
Thus,
$e=e\cdot e \cdot e
= \sum_{i=1}^n e(y_i x z_i) e
= \sum_{i=1}^n (ey_i)(exe)(z_i e)
= \sum_{i=1}^n (ey_ie) x (e z_i e)$.
This shows that $e\in I$.
Using that $e$ is the multiplicative identity element of $eSe$ we conclude that $I=eSe$. The center of any unital simple ring is a field, and hence the last statement also follows. 
\end{proof}

\begin{lem}\label{CornerGenerated}
Let $R$ be a ring, let $e \in R$
be an idempotent, and let $\delta : R \rightarrow R$ be a derivation
such that $\delta(e)=0$.
Consider the non-unital differential polynomial ring $S:=R[x;\delta]$.
The corner subring $eSe$ is generated by $eRe$ and the set $\{ex^ie\}_{i \in \N}$.
\end{lem}

\begin{proof}
Take $a= a_0 + \sum_{i=1}^n a_i x^i \in eSe$, where $a_0,a_1,\ldots,a_n \in R$.
Obviously, $a=eae$ and hence
$a = ea_0e + \sum_{i=1}^n e a_i x^i e$.
Using that $\delta(e)=0$, we notice that for any $i\in \{1,\ldots,n\}$, we have
\begin{displaymath}
	e a_i x^i e = (e a_i x^i e) e
	= \left(
	\sum_{j=0}^i \binom{i}{j} e a_i \delta^{i-j}(e) x^j \right) e
	=
	e a_i e x^i e = (e a_i e) (e x^i e). \qedhere
\end{displaymath}
\end{proof}

An element $d \in R$, of an $s$-unital ring $R$, is said to be an \emph{$s$-unit} for a finite set $F \subseteq R$ if $df=fd=f$ for every $f\in F$.
One can show, that if $R$ is an $s$-unital ring, then every finite subset of $R$ has an $s$-unit.

\begin{lem}\label{lemma:deltad}
Let $R$ be a ring with a derivation $\delta$. If $a \in R$ and $d$ is an $s$-unit for the set $\{a,\delta(a)\}$, then $a\delta(d) =0.$
\end{lem}

\begin{proof}
We get that
$\delta(a) = \delta(ad) = a\delta(d)+\delta(a)d= a\delta(d)+\delta(a)
$.
Thus, $a\delta(d) =0.$
\end{proof}

\begin{defn}\label{def:Hn}
Let $L$ be a non-empty subset of an arbitrary non-unital Ore extension $R[x;\sigma,\delta]$.
For each positive integer $n$ we define
\begin{displaymath}
	H_n(L) := \left\{ r\in R \ \big\lvert \ \exists c_0, c_1, \ldots, c_{n-1} \in R \quad \text{such that} \quad r x^n 
	+ \sum_{i=0}^{n-1} c_i x^i \in L \right\}.
\end{displaymath}
\end{defn}

\begin{lem}\label{lemma:Hn}
Let $R[x;\sigma,\delta]$ be a non-unital Ore extension.
For every positive integer $n$, the following three assertions hold:
\begin{enumerate}[{\rm (a)}]
	\item\label{prop:Ll} If $L$ is a left ideal of $R[x;\sigma,\delta]$, then $H_n(L)$ is a left ideal of $R$;
	\item\label{prop:Lr} If $L$ is a right ideal of $R[x;\sigma,\delta]$, and $\sigma$ is surjective, then $H_n(L)$ is a right ideal of $R$;
	\item\label{prop:Linv} 
	If $R$ is $s$-unital, $\sigma=\id_R$ and $L$ is an ideal of $R[x;\sigma,\delta]$, then $H_n(L)$ is a $\delta$-invariant ideal of $R$.
\end{enumerate}
\end{lem}

\begin{proof}
\eqref{prop:Ll}: This is clear.

\eqref{prop:Lr}: This is clear.

\eqref{prop:Linv}:
By \eqref{prop:Ll} and \eqref{prop:Lr}, clearly $H_n(L)$ is a two-sided ideal of $R$.
Take $a\in H_n(L)$ and a corresponding element
$y:=ax^n + \sum_{i=0}^{n-1} c_i x^i \in L$.
Let $d \in R$ be an $s$-unit for the set $\{a,\delta(a),c_{n-1}\}$.
Notice that $dx y - y dx \in L$.
Using Lemma~\ref{lemma:deltad} (at one step), we compute
\begin{align*}
	dx y - y dx =
	dx \left( ax^n + \sum_{i=0}^{n-1} c_i x^i \right)
	- \left( ax^n  + \sum_{i=0}^{n-1} c_i x^i \right) dx \\
	= da x^{n+1} + d \delta(a) x^n +  d c_{n-1} x^n - \left( a d x^{n+1} + n\cdot a \delta(d) x^n + c_{n-1} d x^{n} \right) + [\text{lower degree terms}]\\
		= a x^{n+1} + \delta(a) x^n +  c_{n-1} x^n - \left( a x^{n+1} + 0 + c_{n-1} x^{n} \right) + [\text{lower degree terms}]\\
	= \delta(a) x^n + [\text{lower degree terms}].
\end{align*}
This shows that $\delta(a) \in H_n(L)$. Thus, $H_n(L)$ is $\delta$-invariant.
\end{proof}

We are now going to present the main result of this article which
generalizes \cite[Theorem 4.15]{ORS2013}.

\begin{thm}\label{thm:simplicitynonunitality}
Let $R$ be an $s$-unital ring and let $\delta : R \rightarrow R$ be a derivation
such that $\ker(\delta)$ contains a nonzero idempotent.
Consider the non-unital differential polynomial ring $S := R[x; \delta]$.
The following five assertions are equivalent:
\begin{enumerate}[{\rm (i)}]
	\item\label{ThmLocal:Simple} $S$ is a simple ring;
	
		\item\label{ThmLocal:EachCorner2} $R$ is $\delta$-simple and $eSe$ is simple for every nonzero idempotent $e\in R$;
	
	\item\label{ThmLocal:EachCorner} $R$ is $\delta$-simple and $Z(eSe)$ is a field for every nonzero idempotent $e\in R$;
	
	\item\label{ThmLocal:SomeCorner2} $R$ is $\delta$-simple and $eSe$ is simple for some nonzero idempotent $e\in \ker(\delta)$;
	
	\item\label{ThmLocal:SomeCorner} $R$ is $\delta$-simple and $Z(eSe)$ is a field for some nonzero idempotent $e\in \ker(\delta)$.
	
\end{enumerate}
\end{thm}

\begin{proof}
\eqref{ThmLocal:Simple}$\Rightarrow$\eqref{ThmLocal:EachCorner2}:
Suppose that \eqref{ThmLocal:Simple} holds.
By Proposition~\ref{sigmadeltasimple} we get that
$R$ is $\delta$-simple.
Take a nonzero idempotent $e\in R$.
It follows from Proposition~\ref{prop:simpleCorner} that $eSe$ is simple.

\eqref{ThmLocal:EachCorner2}$\Rightarrow$\eqref{ThmLocal:EachCorner}:
This follows from the fact that the center of a simple unital ring is a field.

\eqref{ThmLocal:EachCorner}$\Rightarrow$\eqref{ThmLocal:SomeCorner}:
Trivial.

\eqref{ThmLocal:SomeCorner}$\Rightarrow$\eqref{ThmLocal:Simple}:
Suppose that \eqref{ThmLocal:SomeCorner} holds.
By Proposition~\ref{prop:idempotents}, $ReR$ is a nonzero $\delta$-invariant ideal of $R$ and hence $R=ReR$.
Let $I$ be a nonzero ideal of $S$.
We claim that $e\in I$.
If we for a moment assume that the claim holds, then $R=ReR \subseteq I$ from which it follows that $S=I$, using the $s$-unitality of $R$.

Now we show the claim.
Choose $m$ to be the smallest non-negative integer
such that $H_m(I)$ is nonzero.
By Lemma~\ref{lemma:Hn}, $H_m(I)$ is a $\delta$-invariant ideal of $R$.
Hence, by $\delta$-simplicity of $R$, we conclude that $H_m(I)=R$. In particular, $e\in H_m(I)$.
In other words, there is some $y = 
\sum_{i=0}^m y_i x^i \in I$ of degree $m$ such that
$y_m=e$.
If $m=0$, then $e=y \in I$ and we are done.
Seeking a contradiction, suppose that $m>0$.
By multiplying $y$ with $e$ from both the left and the right, and using that $\delta(e)=0$, we may assume that $y_i \in eRe$ for every $i \in \{0,\ldots,m\}$.
Take $r \in R$.
Using that $y_m=e$ we get that $\deg((ere)y-y(ere))<\deg(y)$.
Clearly, $(ere)y-y(ere) \in I$ and hence, by the minimality of $m$, we conclude that
$(ere)y=y(ere)$.
Moreover, $(exe)y-y(exe)\in I$ and using that $\delta(e)=0$ and $exe=ex$,
we compute
\begin{align*}
(exe)y-y(exe)
&=
(ex) \Big( e x^m + y_{m-1} x^{m-1} + [\text{lower degree terms}] \Big)  \\
&
-\Big( e x^m + y_{m-1} x^{m-1} + [\text{lower degree terms}] \Big) (ex) \\
&= ex^{m+1} + e\delta(e) x^{m} + ey_{m-1} x^{m} + [\text{lower degree terms}] \\
& - ex^{m+1} - m \cdot e\delta(e) x^m - y_{m-1}e x^m - [\text{lower degree terms}].
\end{align*}
From the above calculation, and the fact
that $y_i=ey_ie$ which leads to
$(ey_{m-1} - y_{m-1}e)x^m=0$,
we conclude that $\deg((exe)y-y(exe))<\deg(y)$.
Again, by the minimality of $m$, $(exe)y=y(exe)$.
By Lemma~\ref{CornerGenerated}, $eSe$ is generated by $eRe$ and $\{e x^i e\}_{i\in \N}$.
Using that $\delta(e)=0$, we notice that $e x^i e = (exe)^i$, for any positive integer $i$.
Thus, $y$ is contained in the field $Z(eSe)$
but this contradicts the assumption that $m>0$.

\eqref{ThmLocal:SomeCorner2}$\Rightarrow$\eqref{ThmLocal:SomeCorner}:
This is trivial.

\eqref{ThmLocal:EachCorner2}$\Rightarrow$\eqref{ThmLocal:SomeCorner2}:
This is trivial.
\end{proof}

\begin{rem}
Note that we can easily recover Theorem~\ref{thm:unitalsimplicy} by choosing  $e=1_R$ in Theorem~\ref{thm:simplicitynonunitality}.
If one instead relies on Theorem~\ref{thm:unitalsimplicy}, then it is possible to shorten the above proof of (v)$\Rightarrow$(i) by regarding $eSe$ as a unital ring.
\end{rem}

By combining Proposition~\ref{prop:idempotents} with Theorem~\ref{thm:simplicitynonunitality} we get the following.

\begin{cor}\label{commutativesimplicitynonunitality}
Let 
$R$
be a commutative $s$-unital ring 
containing a nonzero idempotent, 
and let $\delta : R \rightarrow R$ be a derivation.
Consider the non-unital differential polynomial ring $S := R[x; \delta]$.
The following three assertions are equivalent:
\begin{enumerate}[{\rm (i)}]
	\item $S$ is a simple ring;
	\item $R$ is $\delta$-simple and $Z(eSe)$ is a field for every nonzero idempotent $e\in R$;
	\item $R$ is $\delta$-simple and $Z(eSe)$ is a field for some nonzero idempotent $e\in R$.
\end{enumerate}
\end{cor}

\begin{rem}\label{rem:innerDerivations}
(a)
Recall that a derivation $\delta : R \to R$ is said to be \emph{inner}
if there is some $a\in R$ such that $\delta(r)=ar-ra$, for all $r\in R$.

(b)
If $\delta$ is an inner derivation on a ring $R$, then every ideal of $R$ is $\delta$-invariant.
Consequently, $R$ is $\delta$-simple if and only if $R$ is simple.

(c)
If $R$ is a unital ring and $\delta$ is an inner derivation, then $R[x;\delta]$ is not simple by a result of Goodearl  \cite[Lemma 1.5]{Goodearl92}. This can be proven by noting that if $\delta(r) =ar-ra$ for all $r\in R$, then the proper left ideal generated by $x-a$ is also a right ideal. 
\end{rem}

\begin{cor}\label{cor:NonSimple1}
Let $R$ be a ring with an inner derivation $\delta$, and suppose there exists a nonzero idempotent $e \in R$ such that $\delta(e)=0$. Then $S:=R[x;\delta]$ is not simple.
\end{cor}

\begin{proof} 
The restriction, $d$, of $\delta$ to $eRe$ is an inner derivation by Proposition~\ref{prop:idempotents}. By  Proposition~\ref{prop:cornerIso}, $eSe$ is isomorphic to $eRe[x;d]$. Since $d$ is inner, $eSe$ is not simple by Remark \ref{rem:innerDerivations}. Thus, by Proposition~\ref{prop:simpleCorner}, $S$ is not simple. 
\end{proof}

\begin{cor}\label{cor:innerNotSimple}
Let $R$ be a ring with an inner derivation $\delta$.
The following two assertions hold:
\begin{itemize}
    \item[(a)] If $R$ contains a nonzero central idempotent, then $R[x;\delta]$ is not simple.

    \item[(b)] If $R$ is locally unital, then $R[x;\delta]$ is not simple.
\end{itemize}
\end{cor}

\begin{proof}
In both cases we will invoke Corollary~\ref{cor:NonSimple1}
to reach the desired conclusion.

(a) If $e \in Z(R)$, then $\delta(e)=0$ by Proposition~\ref{prop:idempotents}.

(b) Suppose that there is some $a\in R$ such that $\delta(r)=ar-ra$ for every $r\in R$.
Without loss of generality, we may assume that $a\neq 0$.
By local unitality, we can find a nonzero idempotent $e$ such that $e a = a = a e$.
Clearly, $\delta(e)=ae-ea=0$.
\end{proof}

\section{Examples}\label{sec:ex}

In this section we will illustrate our results by applying them to several classes of examples.

\begin{exmp}[Outer derivation on a non-commutative ring]\label{ex:nonunital}

(a) Let $T$ be a locally unital (nonzero) simple ring such that $nt \neq 0$ for all $n\in \Z \setminus \{0\}$ and $t\in T \setminus \{0\}$, 
and consider the polynomial ring $R:=T[y]$.
Notice that $R$ has a set of local units $E$ which are contained in $T$.
We may define a derivation $\delta : R \to R$ as the $T$-linear extension of the rule
$\delta(t y^i)=i t y^{i-1}$ and $\delta(T)=\{0\}$.
In particular, we notice that $E \subseteq \ker(\delta)$.
Consider the non-unital differential polynomial ring $S := R[x; \delta]$. Our goal is to show that $S$ is a simple ring, by using Theorem~\ref{thm:simplicitynonunitality}.

We begin by showing that $R$ is $\delta$-simple.
Suppose that $I$ is a nonzero $\delta$-invariant ideal of $R$.
Let $a \in I$ be a nonzero element of smallest possible degree.
Clearly the degree of $\delta(a) \in I$ is smaller than that of $a$.
Hence, by our assumption, the degree of $a$ must be $0$, i.e. $a\in T$.
Hence $I \cap T$ is a nonzero ideal of $T$ and using that $T$ is simple, we conclude that $I \cap T = T$. Thus, $I=R$.

Now, pick a nonzero local unit $e \in T \subseteq R$.
We will show that $Z(eSe)$ is a field.
Using that $\delta(e)=0$, we get that
any element $e \left(a_0 + \sum_{i=1}^n a_i x^i\right) e \in eSe$ may be written on the form
$e a_0 e + \sum_{i=1}^n e a_i e x^i$ where $e a_i e \in e R e$.
Take an arbitrary $ec_0e + \sum_{i=1}^m e c_i e x^i \in Z(eSe)$.
Let $b =ete \in eTe \subseteq eSe$ be arbitrary.
Then
\begin{align*}
0 &= \left( ec_0e + \sum_{i=1}^m e c_i e x^i \right) b - b \left( ec_0e + \sum_{i=1}^m e c_i e x^i \right) \\
&= \left( ec_0e + \sum_{i=1}^m e c_i e x^i \right) ete - ete \left( ec_0e + \sum_{i=1}^m e c_i e x^i \right) \\
&= \left( ec_0ete + \sum_{i=1}^m e c_i e te x^i  \right) - \left( etec_0e + \sum_{i=1}^m ete c_i e x^i \right) \\
&= (ec_0ete-etec_0e)   + \sum_{i=1}^m (e c_i e \ e te - ete \ e c_i e) x^i.
\end{align*}
This shows that $e c_i e \in C_{eRe}(eTe)$ for each $i \in \{1,\ldots,m\}$.

Consider the element $e x e \in eSe$ and notice that
\begin{align*}
0 &= \left( ec_0e + \sum_{i=1}^m e c_i e x^i \right) exe - exe \left( ec_0e + \sum_{i=1}^m e c_i e x^i \right) \\
&= \left( ec_0 ex + \sum_{i=1}^m e c_i e x^{i+1} \right) - \left( exec_0e + \sum_{i=1}^m (ex e c_i e) x^i \right) \\
&= \left( ec_0 ex + \sum_{i=1}^m e c_i e x^{i+1} \right) - \left( ec_0ex + \delta(ec_0e) + \sum_{i=1}^m (e c_i e x + \delta(e c_i e) ) x^i \right) \\
&= - \left(\delta(ec_0e) + \sum_{i=1}^m \delta(e c_i e) x^i \right).
\end{align*}
Using that $\delta(e c_i e) =0$, we conclude that $e c_i e \in T$, for each $i\in \{1,\ldots,m\}$.
In conclusion, for each $i\in \{1,\ldots,m\}$, $e c_i e \in Z(eTe)$.
Using that $T$ is simple, it is easy to see that $eTe$ is simple. In particular, $Z(eTe)$ is a field.

Finally, we notice that
\begin{align*}
0 &= \left( ec_0e + \sum_{i=1}^m e c_i e x^i \right) eye - eye \left( ec_0e + \sum_{i=1}^m e c_i e x^i \right) \\
&= \left( ec_0eye + \sum_{i=1}^m e c_i e \left( \delta^i(eye) + \sum_{j=1}^i \binom{i}{j} \delta^{i-j}(eye) x^j \right)\right)
- \left( eyec_0e + \sum_{i=1}^m eye c_i e x^i \right) \\
&= \left( ec_1 e ( eyex + \delta(eye)) + \sum_{i=2}^m e c_i e \left( 1 \cdot eye x^i + i \cdot \delta(eye) x^{i-1} \right) \right)
- \left( \sum_{i=1}^m eye c_i e x^i \right) \\
&= ec_1e + \sum_{i=2}^m e c_i e \cdot i \cdot \delta(eye) x^{i-1}
= ec_1e + \sum_{i=2}^m i \cdot e c_i e x^i.
\end{align*}
This shows that $ec_ie = 0$ for each $i \in \{1,\ldots,m\}$.
In conclusion, we have shown that $Z(eSe)=Z(eTe)$ which is a field.

(b) More concretely, in (a) we may e.g. take $T:=M_\infty(\R)$, the ring of infinite matrices with a finite number of nonzero entries. It is easy to see that $T$ is a non-unital simple ring
and that it is locally unital.
\end{exmp}

\begin{exmp}[Inner derivation on a non-commutative ring]\label{ex:innernoncomm}
Consider $R:=M_\infty(\R)$. Take any nonzero $m \in R$ and define an inner derivation
$\delta(r) = mr-rm$.
Clearly, a set of local units for $R$ can be formed by taking all diagonal matrices with finitely many ones on the diagonal and the rest of the entries being zero. By Corollary \ref{cor:innerNotSimple}, $S:=R[x;\delta]$ is not simple.
\end{exmp}

\begin{exmp}[Outer derivation on a commutative ring]\label{ex:outercomm}
Let $T$ be the algebra of functions on $\R$ that are compactly supported and smooth except at a finite number of points, i.e. any function $ f \in T$ has compact support and there exist points $x_1, x_2, \ldots x_n$ such that $f$ is $C^{\infty}$ on $\R \setminus \{x_1, x_2. \ldots, x_n\}$. The set of functions in $T$ that are zero almost everywhere (equivalently, zero except in a finite number of points) form an ideal, $I$. Define $R := T/I.$ The elements in $R$ are equivalence classes of functions from $T$, consisting of functions that are equal except in a finite number of points. 

We can use the usual derivative from calculus to define a derivation $\delta$ on $R$ in the obvious way. Indeed, if $f$ and $g$ are functions in $T$ belonging to the same equivalence class, then they are equal and smooth in a neighbourhood of almost every point. Hence, their derivatives $f'$ and $g'$ (definining them arbitrarily where $f$ and $g$ are not differentiable) belong to the same equivalence class in $T$. It is clear that the product and sum rules hold.

Clearly, $R$ has a set $E$ of local units consisting of characteristic functions.
Obviously, $E \subseteq \ker(\delta)$, but one may also use Proposition~\ref{prop:idempotents} and commutativity.
Let $J$ be the ideal of $R$ consisting of the equivalence classes of all functions
vanishing outside of the interval $[0,1]$. Clearly, $J$ is a proper $\delta$-invariant ideal of $R$.
Thus, by Corollary~\ref{commutativesimplicitynonunitality}, $S:=R[x;\delta]$ is not simple.
\end{exmp}

\end{document}